\documentclass[letterpaper, 10 pt, conference]{ieeeconf}  

\IEEEoverridecommandlockouts                              

\usepackage{amsthm,amsmath,amsfonts,mathtools,enumerate,hyperref,dsfont,tikz,hhline,array}
\usepackage[american]{circuitikz}

\ctikzset{bipoles/length=0.7cm, bipoles/thickness=1}
\newcommand\esymbol[1]{\begin{circuitikz}
\draw (0,0) to [#1] (0.6,0); \end{circuitikz}}
\allowdisplaybreaks

\newcommand{\tr}{\mathop{\bf tr}}

\newcommand{\prox}{\mathbf{prox}}
\newcommand{\proj}{\mathbf{proj}}
\newcommand{\lm}{\mathbf{lm}}

\newcommand{\argmin}{\mathop{\rm argmin}}

\newcommand{\norm}[1]{\left\lVert#1\right\rVert}
\newcommand{\mnorm}[1]{{\left\vert\kern-0.25ex\left\vert\kern-0.25ex\left\vert #1 
    \right\vert\kern-0.25ex\right\vert\kern-0.25ex\right\vert}}

\newtheorem{theorem}{Theorem}
\newtheorem{lemma}{Lemma}
\newtheorem{corollary}{Corollary}
\newtheorem{remark}{Remark}

\newtheorem{assumption}{Assumption}

\newcommand{\ie}{{\it i.e.}}

\title{\LARGE \bf RC Circuits based Distributed Conditional Gradient Method
}

\author{Yue~Yu and Beh\c{c}et~A\c{c}\i kme\c{s}e
\thanks{
The authors are with the Department of Aeronautics and Astronautics, University of Washington, Seattle,
        WA 98195 USA (emails: 
        {{\tt\small yueyu@uw.edu,behcet@uw.edu})}.}%
}

\begin{document}

\maketitle
\thispagestyle{empty}
\pagestyle{empty}

\begin{abstract}
We consider distributed optimization on undirected connected graphs.
We propose a novel distributed conditional gradient method with \(O(1/\sqrt{k})\) convergence. Compared with existing methods, each iteration of our method uses both communication and linear minimization step only once rather than multiple times. We further extend our results to cases with composite local constraints. We demonstrate our results via examples on distributed matrix completion problem.   
\end{abstract}

\section{Introduction}
Distributed optimization aims to optimize sum of convex functions via only local computation and communication on a connected graph \(\mathcal{G}\) with node set \(\mathcal{V}\) and edge set \(\mathcal{E}\) \cite{bertsekas1989parallel,boyd2011distributed}, which takes the following form,
\begin{equation}
    \begin{array}{ll}
    \underset{x_1, \ldots, x_{|\mathcal{V}|}}{\mbox{minimize}}  & \sum_{i\in\mathcal{V}} f_i(x_i) \\
    \mbox{subject to}  &  x_i=x_j, \enskip \forall \{ij\}\in \mathcal{E},\\
    & x_i\in X_i, \enskip \forall i\in\mathcal{V}.
    \end{array}\label{opt: template}
\end{equation}
where \(f_i\) and \(X_i\) are respectively the convex objective function and convex feasible set available to node \(i\) only. Problem \eqref{opt: template} arises frequently from multi-agent applications such as distributed tracking and estimation  
\cite{li2002detection,lesser2012distributed,accikmecse2014decentralized}.

Many distributed optimization algorithms have been developed to solve \eqref{opt: template}. To ensure each local variable \(x_i\) remains in constraint set \(X_i\), different algorithms use different \emph{oracles}--computation subroutines called at each iteration--on each node \(i\in\mathcal{V}\). We list three of the most popular ones, where \(c_i\) denotes certain constants:
\begin{subequations}
\begin{align}
    \prox(c_i)=&\argmin_{x_i\in X_i}\,\, f_i(x_i)+\textstyle\frac{\rho}{2}\norm{x_i-c_i}_2^2,\label{oracle: prox} \\
    \proj(c_i)=&\argmin_{x_i\in X_i}\,\, \textstyle\frac{1}{2}\norm{x_i-c_i}_2^2,\label{oracle: proj}\\
    \lm(c_i)=&\argmin_{x_i\in X_i}\,\, \langle c_i, x_i\rangle.\label{oracle: lmo}
\end{align}
\end{subequations}

The first one is the \emph{proximal oracle} given by \eqref{oracle: prox}  with \(\rho>0\), which is widely used by distributed Alternating Direction Method of Multipliers (ADMM) \cite{wei2012distributed,shi2014linear,meng2015proximal}. Since exact evaluation of a proximal oracle may require an iterative algorithm itself, distributed ADMM is usually difficult to implement. A more efficient alternative is the \emph{projection oracle} given by \eqref{oracle: proj}. Such an oracle minimizes the quadratic distance to a reference point \(c_i\). Typical algorithms using projection oracle are distributed projected subgradient methods  \cite{nedic2010constrained,ram2010distributed,xi2016distributed}. We note that the quadratic function in proximal and projection oracles can be further generalized to Bregman divergence of strongly convex function \cite{duchi2012dual,li2016distributed,yuan2018optimal,yu2018bregman,doan2019convergence,yu2019stochastic}.

Another oracle that recently become popular is the \emph{linear minimization oracle} given by \eqref{oracle: lmo}, which, instead of the quadratic function in \eqref{oracle: proj}, optimizes a linear function. First proposed in conditional gradient method (a.k.a Frank--Wolfe method) \cite{frank1956algorithm}, such oracle lately received renewed interest due to its computation efficiency over the convex hulls of an atomic set \cite{clarkson2010coresets,jaggi2013revisiting,bach2012equivalence}. Algorithms that solve problem \eqref{opt: template} using oracle \eqref{oracle: lmo} are commonly known as \emph{distributed conditional gradient method} \cite{lafond2016d,wai2017decentralized,zheng2018distributed,li2018communication}. 

However, the existing distributed conditional gradient methods have the following limitations. The algorithm in \cite{zheng2018distributed} assumes the underlying graph \(\mathcal{G}\) has a master-slave hierarchy, which is sensitive to node failure. The algorithms proposed in \cite{lafond2016d,wai2017decentralized,li2018communication} relaxed this assumption, but each iteration of the resulting algorithm either uses multiple (at least two) communication steps until a consensus condition is reached \cite{lafond2016d,wai2017decentralized}, or multiple linear minimization steps until an optimality condition is reached \cite{li2018communication}.  These observations motivate the following question:

\emph{Is it possible to design a distributed conditional gradient method that uses both communication and linear minimization step only once per iteration?}

In this work, we answer this question affirmatively and make the following contributions. 

\begin{enumerate}
    \item We propose a novel distributed conditional gradient method with \(O(1/\sqrt{k})\) convergence. Compared with existing methods \cite{wai2017decentralized,li2018communication}, our method uses both communication and linear minimization step once per iteration rather than multiple times, and allows approximate linear minimization oracles.
    \item We further extend our results to problems with composite local constraints by combining linear minimization oracle and projection oracle together, which allows more efficient computation than either oracle alone. 
\end{enumerate}

Our work combines ideas from conditional-gradient based augmented Lagrangian methods \cite{yurtsever2018conditional,yurtsever2019conditional} and physics inspired distributed algorithms \cite{yu2018mass,yu2019rlc}. The rest of the paper is organized as follows. After the preliminaries in Section~\ref{sec: preliminaries}, we present our algorithm and its convergence proof in Section~\ref{sec: method} and Section~\ref{sec: convergence}, then further extend them to composite constraint case in Section~\ref{sec: extension}. We then demonstrate our results via numerical examples in Section~\ref{sec: experiment}  before conclude in Section~\ref{sec: conclusion}.

\section{Preliminaries}\label{sec: preliminaries}
Let \(\mathbb{R}\) denote the real numbers, \(\mathbb{R}^n\) the \(n\)-dimensional real numbers. We use \(\cdot^\top\) to denote matrix (and vector) transpose. Let \(\langle x, y\rangle= x^\top y\) and \(\norm{x}_2=\sqrt{\langle x, x\rangle}\) denote the inner product and, respectively, its induced norm. Let \(I_n\) denotes the \(n\times n\) identity matrix and \(\otimes\) denote Kronecker product.

\subsection{Graph theory}
An undirected graph \(\mathcal{G}=(\mathcal{V}, \mathcal{E})\) consists of a node set \(\mathcal{V}\) and an edge set \(\mathcal{E}\), where an edge is a pair of distinct nodes in \(\mathcal{V}\). For an arbitrary orientation on \(\mathcal{G}\), \ie , each edge has a head and a tail, the \(|\mathcal{V}|\times |\mathcal{E}|\) incidence matrix is denoted by \(\overline{E}(\mathcal{G})\). The columns of \(\overline{E}(\mathcal{G})\) are indexed by the edges in \(\mathcal{E}\), and the entry on their \(i\)-th row takes the value ``\(1\)" if node \(i\) is the head of the edge, ``\(-1\)" if it is its tail, and \(0\) otherwise.
When graph \(\mathcal{G}\) is connected, the nullspace of \(\overline{E}(\mathcal{G})^\top\) is spanned by vector of all \(1\)'s.

\subsection{Convex Analysis} 
Let \(X\subseteq \mathbb{R}^n\) denote a closed convex set. 
A continuously differentiable function  \(f: \mathbb{R}^n\to\mathbb{R}\) is convex if and only if, for all \( x, x'\in \mathbb{R}^n,\)
\begin{equation}
    f(x')\geq f(x)+\langle \nabla f(x), x'-x\rangle.
    \label{eqn: convexity}
\end{equation}
We say a convex function \(f\) is \(\beta\)-smooth if \(\frac{\beta}{2}\norm{\cdot}_2^2-f\) is also convex, which implies the following \cite[Thm. 2.1.5]{nesterov2018lectures}
\begin{subequations}\label{eqn: smooth}
\begin{align}
      f(x')\leq & \textstyle f(x)+\langle\nabla f(x), x'-x\rangle+\frac{\beta}{2}\norm{x'-x}_2^2, 
    \label{eqn: smooth 1}\\
     f(x') \geq & \textstyle f(x)+\langle \nabla f(x), x'-x\rangle+\frac{1}{2\beta}\norm{\nabla f(x')-\nabla f(x)}_2^2.\label{eqn: smooth 2}
    \end{align}
\end{subequations}The normal cone \(N_X(x)\) at \(x\in X\) is given by
\begin{equation}\label{def: normal cone}
  N_{X}(x)=\{u|\, \langle u, x'-x\rangle\leq 0, \forall x'\in X\}.
\end{equation}
The projection map \(P_X\) onto set \(X\) is given by
\begin{equation}
 P_X(x)=\textstyle \argmin_{y\in X} \textstyle\frac{1}{2}\norm{x-y}^2_2.
    \label{eqn: projection}
\end{equation}

\section{Algorithm}
\label{sec: method}
We present our main algorithm in this section, which is inspired by RC circuits dynamics. Throughout, we assume graph \(\mathcal{G}=(\mathcal{V}, \mathcal{E})\) is undirected and connected. We also define the following matrices based on graph \(\mathcal{G}\)
\begin{equation}
    E(\mathcal{G})=\overline{E}(\mathcal{G})\otimes I_n,\enskip  L(\mathcal{G})=(\overline{E}(\mathcal{G})\overline{E}(\mathcal{G})^\top)\otimes I_n.
\end{equation}
With these definitions, we can rewrite the optimization template \eqref{opt: template} in the following form
\begin{equation}
    \begin{array}{ll}
        \underset{x}{\mbox{minimize}} & f(x)=\sum_{i\in\mathcal{V}} f_i(x_i) \\
        \mbox{s.t.} & E(\mathcal{G})^\top x=0, \enskip x\in X=\prod_{i=1}^{|\mathcal{V}|}X_i
    \end{array}\label{opt: main}
\end{equation}
where \(x=[x_1^\top, \ldots, x_{|\mathcal{V}|}^\top]^\top\) and \(\prod_{i=1}^{|\mathcal{V}|}X_i\) is the Cartesian product of \(X_1, \ldots, X_{|\mathcal{V}|}\). We assume, for all \(i\in\mathcal{V}\), that \(f_i\) is convex and differentiable, \(X_i\subset\mathbb{R}^n\) is convex and compact.
\begin{figure}[ht]
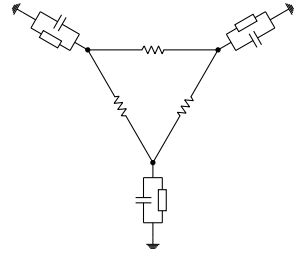

    \centering
    \ctikzset{bipoles/length=0.7cm, bipoles/thickness=1}
\begin{tikzpicture}[scale=0.5]
   
\begin{scope}[rotate=0,transform shape]
    \input{method/RC/subRC.tex}
\end{scope}

\begin{scope}[rotate=120,transform shape]
    \input{method/RC/subRC.tex}
\end{scope}

\begin{scope}[rotate=-120,transform shape]
    \input{method/RC/subRC.tex}
\end{scope}

\end{tikzpicture}
    \caption{An illustration of RC circuits}
    \label{fig: RC}
\end{figure}

Aiming to design a distributed algorithm for problem \eqref{opt: main}, we consider a conceptual RC circuits model defined on graph \(\mathcal{G}=(\mathcal{V}, \mathcal{E})\) as follows. Let each node \(i\in\mathcal{V}\) denote a pin with electrical potential \(x_i(t)\) at time \(t\). We add a linear capacitor with unit capacitance in parallel with a nonlinear resistor which maps potential \(x_i\) to \(\nabla f_i(x_i)\) between each pin \(i\) and ground (zero potential point), then add a linear resistor with time varying resistance \(1/r(t)\) on each edge \(\{ij\}\in\mathcal{E}\) where \(r(t)>0\) and \(\lim_{t\to\infty}1/r(t)=0\). See Fig.~\ref{fig: RC} and Table~\ref{tab:RC} for an illustration.
\begin{table}[ht]
\caption{Voltage-current relation of RC units}
    \label{tab:RC}
    \centering
    \ctikzset{bipoles/length=0.42cm}
\begin{tabular}{|c|c|c|c|}
    \hline
    \small
    type & symbol & voltage & current \\
    \hhline{====}
  non-linear resistor & \esymbol{european resistor} & \(x_i\) & \(\nabla f_i(x_i)\)\\
   linear resistor & \esymbol{resistor} & \(x_i-x_j\) & \(r (x_i-x_j)\)\\
   capacitor & \esymbol{capacitor} & \(x_i\) & \(\frac{d}{dt}x_i\)\\
    \hline
    \end{tabular}
\end{table} 
As the resistance on edges decreases to zero, \ie, \(1/r(t)\to 0\), the potential value on neighboring nodes necessarily reaches the same, \ie, 
\begin{equation}
    x_i(t)-x_j(t)\to 0, \enskip \forall \{ij\}\in\mathcal{E}.\label{eqn: equi KVL}
\end{equation}
Further, if the circuits are reaching an equilibrium where \(\frac{d}{dt}x_i(t)\to 0\) for all \(i\in\mathcal{V}\), then applying Kirchoff current law to the collection of all edges gives
\begin{equation}
    \textstyle\sum_{i\in\mathcal{V}}\nabla f_i(x_i(t))\to 0. \label{eqn: equi KCL}
\end{equation}
Notice that conditions in \eqref{eqn: equi KVL} and \eqref{eqn: equi KCL} are exactly the optimality conditions of \eqref{opt: main} when \(X=\mathbb{R}^{|\mathcal{V}|n}\), which suggests that the dynamics of the constructed RC circuits may provide a prototype algorithm for problem \eqref{opt: main}. Following this intuition, we apply Kirchoff current law to pin \(i\in\mathcal{V}\) for any \(t\geq 0\) (not necessarily at equilibrium) and obtain
\begin{equation}
    \textstyle 0=\frac{d}{dt}x_i(t)+\nabla f_i(x_i(t))+r(t)\sum_{j\in\mathcal{N}(i)}(x_j(t)-x_i(t)),\label{eqn: KCL pin}
\end{equation}
where \(j\in\mathcal{N}(i)\) if and only if \(\{ij\}\in\mathcal{E}\). Let \(x(t)=[x_1(t)^\top, \ldots, x_{|\mathcal{V}|}(t)^\top]^\top\) and \(f(x)=\sum_{i\in\mathcal{V}}f_i(x_i)\), then we can rewrite \eqref{eqn: KCL pin} for all \(i\in\mathcal{V}\) compactly as follows
\begin{equation}
   \textstyle\frac{d}{dt}x(t)=-\nabla f(x(t))-r(t)L(\mathcal{G})x(t),
    \label{eqn: KCL}
\end{equation}
A naive Euler-forward discretization of \eqref{eqn: KCL} says that \(x^{k+1}\) is obtained by moving \(x^k\) in the direction of \(-\nabla f(x^k)-r^kL(\mathcal{G})x^k\). However, it is difficult to choose appropriate step sizes so that \(x^k\in X\) for all \(k\). To remedy this, we propose the following discretization of \eqref{eqn: KCL}
\begin{equation}
    \begin{aligned}
    y^k=&\argmin_{y\in X}\langle \nabla f(x^k)+r^k L(\mathcal{G})x^k, y\rangle,\\
    x^{k+1}=&x^k+\alpha^k(y^k-x^k),
    \end{aligned}\tag{RC}\label{alg: RC}
\end{equation}
where \(\alpha^k\in(0, 1]\). Iteration \eqref{alg: RC} says that \(x^{k+1}\) is obtained by moving \(x^k\) towards \(y^k\), which is the extreme point when moving in the direction \(-\nabla f(x^k)-r^kL(\mathcal{G})\) without leaving set \(X\). Notice that \(x^{k+1}\) is a convex combination of \(x^k\) and \(y^k\) as \(\alpha^k\in (0, 1]\). Since \(X\) is convex, this ensures \(x^{k+1}\in X\) whenever \(x^k\in X\). Hence, algorithm \eqref{alg: RC} ensures \(x^k\in X\) for all \(k\geq 1\) as long as \(x^1\in X\). 

\begin{remark}
Compared with existing methods, each iteration in algorithm~\eqref{alg: RC} uses both communication and linear minimization step only once, which is more efficient than the multiple (at least two) communication steps  in \cite{wai2017decentralized} and multiple linear minimization steps in \cite{li2018communication}.
\end{remark}

In the next section, we will show that algorithm \eqref{alg: RC} indeed converges to the optimum of \eqref{opt: main}. The key challenge is to find an appropriate sequence of step sizes \(\{\alpha^k\}\) and determine how fast the sequence \(\{r^k\}\) grows to infinity.

\section{Convergence}
\label{sec: convergence}
In this section, we establish the convergence of algorithm \eqref{alg: RC} proposed in the previous section, and further extend it to cases with approximate linear minimization. We first group our technical assumptions as follows.
\begin{assumption}
\begin{enumerate}
    \item Graph \(\mathcal{G}=(\mathcal{V}, \mathcal{E})\) is undirected and connected.
    \item  For all \(i\in\mathcal{V}\), \(f_i: \mathbb{R}^n\to \mathbb{R}\) is continuously differentiable, convex and \(\beta\)-smooth, \ie, both \(f_i\) and \(\frac{\beta}{2}\norm{\cdot}_2^2-f_i\) are convex, \(X_i\subseteq\mathbb{R}^n\) is a compact convex set. We assume \(\max_{x, x'\in X}\norm{x-x'}^2_2\leq \delta\) for some \(\delta>0\) where \(X=\prod_{i=1}^{\mathcal{V}}X_i\).
\end{enumerate}\label{asp: basic}
\end{assumption}

\begin{assumption}\label{asp: KKT}
 There exists \(x^\star\in X\) and \(u^\star\) such that
\begin{subequations}
    \begin{align}
    E(\mathcal{G})^\top x^\star=&0,\label{kkt: primal}\\
    -E(\mathcal{G})u^\star-\nabla f(x^\star)\in & N_{X}(x^\star).\label{kkt: dual}
    \end{align}
\end{subequations}
\end{assumption}
Based on these assumptions, the following theorem shows the convergence of \eqref{alg: RC} in terms of both the objective function value and consensus error, where we let \(\norm{L(\mathcal{G})}_2\) to denote the largest eigenvalue of \(L(\mathcal{G})\) (all proofs are delayed to the Appendix).
\begin{theorem}\label{thm: 1}
Suppose Assumption~\ref{asp: basic} and \ref{asp: KKT} hold. If sequence \(\{x^k\}\) is generated by \eqref{alg: RC} with \(\alpha^k=\frac{2}{k+1}\), \(r^k=r^0\sqrt{k+1}\) for some \(r^0>0\), then
\begin{equation*}
\begin{aligned}
    |f(x^k)-f(x^\star)|\leq & \textstyle\frac{2\max\{\sigma^k \delta^2, \rho^2+\rho\sqrt{\sigma^k\delta}\}}{\sqrt{k}},\\
    \norm{E(\mathcal{G})^\top x^k}_2\leq & \textstyle\frac{2(\rho+\sqrt{\sigma^k\delta})}{ \sqrt{k}},
    \end{aligned}
\end{equation*}
where \(\rho=\norm{u^\star}_2\) and \(\sigma^k=\frac{\beta}{\sqrt{k}}+\norm{L(\mathcal{G})}_2r^0\).
\end{theorem}

Each iteration of \eqref{alg: RC} requires an exact linear minimization. However, it is important to note that conditional gradient method itself is known to be robust to approximate linear minimization as well \cite{jaggi2013revisiting}. If we let \(y^\epsilon\) be an \emph{\(\epsilon\)-optimal solution} to \(\min_{y\in X}\langle c, y\rangle\) in the following sense 
\begin{equation}
    \langle c, y^\epsilon\rangle -\min_{y\in X}\langle c, y\rangle\leq\epsilon, \enskip y^\epsilon\in X.\label{eqn: epsilon optimal}
\end{equation}
Then the following corollary shows that if the linear minimization in algorithm \eqref{alg: RC} is solved approximatedly in the sense of \eqref{eqn: epsilon optimal} with increasing accuracy, then convergence results similar to those in Theorem~\ref{thm: 1} still hold.

\begin{corollary}\label{cor: 1}
If the \(y^k\) used in \eqref{alg: RC} is replaced by an \(\epsilon^k\)-optimal solution to the corresponding linear minimization in the sense of \eqref{eqn: epsilon optimal}, then
Theorem~\ref{thm: 1} still holds with \(\sigma^k=(1+\kappa)(\frac{\beta}{\sqrt{k}}+\norm{L(\mathcal{G})}_2r^0)\) for some \(\kappa>0\) if 
\[\textstyle\epsilon^k\leq \kappa\big(\frac{\beta}{\sqrt{k+1}}+\norm{L(\mathcal{G}}_2r^0\big)\frac{\delta}{\sqrt{k+1}}.\] 
\end{corollary}

\section{Distributed Optimization with Composite Local Constraints}\label{sec: extension}
One limitation of existing distributed conditional gradient methods \cite{lafond2016d,wai2017decentralized,zheng2018distributed,li2018communication} is that their iterates require linear minimization over the entire local constraint set, which can be computationally challenging. It also completely discard projection oracles, which may lead to efficient computation in many interesting scenarios \cite{nedic2010constrained,ram2010distributed,xi2016distributed}. Motivated by these observations, we consider the following extension to problem \eqref{opt: main} with \emph{composite constraints} on each node  
\begin{equation}
    \begin{array}{ll}
        \underset{x}{\mbox{minimize}} & f(x)=\sum_{i\in\mathcal{V}} f_i(x_i) \\
        \mbox{s.t.} & E(\mathcal{G})^\top x=0, \\
        &x\in X\cap Y=(\prod_{i=1}^{|\mathcal{V}|}X_i)\cap (\prod_{i=1}^{|\mathcal{V}|}Y_i)
    \end{array}\label{opt: composite}
\end{equation}
where, in addition to the assumptions we made for \eqref{opt: main}, we assume \(Y_i\subseteq\mathbb{R}^n\) is a close convex set; we also assume that the projection oracle \eqref{oracle: proj} is efficient on \(Y_i\), \ie, \(P_{Y_i}(x_i)\) is easy to compute for all \(i\in\mathcal{V}\). 
 
To exploit the structure of \eqref{opt: composite}, we propose the following modification to \eqref{alg: RC}, whose linear minimization contains a penalty term for not only the consensus constraints violation \(L(\mathcal{G})x^k\) but also the difference between \(x^k\) and \(P_Y(x^k)\). 
\begin{equation}
    \begin{aligned}
    y^k=&\argmin_{y\in X}\langle \nabla f(x^k)+r^k (x^k-P_Y(x^k)+L(\mathcal{G})x^k), y\rangle,\\
    x^{k+1}=&x^k+\alpha^k(y^k-x^k).
    \end{aligned}\tag{RC-co}\label{alg: RC composite}
\end{equation}
Since \(Y=\prod_{i=1}^{|\mathcal{V}|} Y_i\), it is straightforward to show that
\begin{equation}\label{eqn: projection decomp}
P_Y(x)=[P_{Y_1}(x_1)^\top, \ldots, P_{Y_{|\mathcal{V}|}}(x_{|\mathcal{V}|})^\top]^\top,
\end{equation}
for any \(x=[x_i^\top, \ldots, x_{|\mathcal{V}|}^\top]^\top\).
Hence \eqref{alg: RC composite} also allows fully distributed implementation.

We now prove that, with proper modifications to Assumption~\ref{asp: KKT}, the results similar to those in Section~\ref{sec: convergence} still hold (all proofs are delayed to the Appendix)..
\begin{assumption}\label{asp: KKT composite} For all \(i\in\mathcal{V}\), \(Y_i\subseteq \mathbb{R}^n\) is a compact convex set. There exists \(x^\star\in X\), \(u^\star\) and \(v^\star\in N_Y(x^\star)\) such that
\begin{subequations}
    \begin{align}
    E(\mathcal{G})^\top x^\star=&0,\label{kkt composite: primal}\\
    -E(\mathcal{G})u^\star-\nabla f(x^\star)-v^\star\in & N_{X}(x^\star).\label{kkt composite: dual}
    \end{align}
\end{subequations}
\end{assumption}

\begin{theorem}\label{thm: 2}
Suppose Assumption~\ref{asp: basic} and \ref{asp: KKT composite} hold. If sequence \(\{x^k\}\) is generated by \eqref{alg: RC composite} with \(\alpha^k=\frac{2}{k+1}\), \(r^k=r^0\sqrt{k+1}\) for some \(r^0>0\), then
\begin{equation*}
\begin{aligned}
    |f(x^k)-f(x^\star)|\leq & \textstyle\frac{2\max\{\sigma^k\delta^2, \rho^2+\rho\sqrt{\sigma^k\delta}\}}{\sqrt{k}},\\
    \norm{\begin{bmatrix}
E(\mathcal{G})^\top x^k\\
x^k-P_Y(x^k)
\end{bmatrix}}_2\leq &\textstyle \frac{2(\rho+\sqrt{\sigma^k\delta})}{\sqrt{k}},
    \end{aligned}
\end{equation*}
where \(\rho = \norm{\begin{bmatrix}
u^\star\\
v^\star
\end{bmatrix}}_2\) and \(\sigma^k=\frac{\beta}{\sqrt{k}}+(\norm{L(\mathcal{G})}_2+1)r^0\).
\end{theorem}
\begin{remark}\label{rem: equivalence}
Algorithm~\eqref{alg: RC composite} can also be interpreted as a combination of Nestrov smoothing of indicator function and conditional gradient method \cite{yurtsever2018conditional}.  
\end{remark}

\begin{corollary}\label{cor: 2}
If the \(y^k\) used in \eqref{alg: RC composite} is replaced by an \(\epsilon^k\)-optimal solution to to the corresponding linear optimization in the sense of \eqref{eqn: epsilon optimal}, then
Theorem~\ref{thm: 2} still holds with \(\sigma^k=(1+\kappa)(\frac{\beta}{\sqrt{k}}+(\norm{L(\mathcal{G})}_2+1)r^0)\) for some \(\kappa>0\) if 
\[\textstyle\epsilon^k\leq \kappa\big(\frac{\beta}{\sqrt{k+1}}+(\norm{L(\mathcal{G}}_2+1)r^0\big)\frac{\delta}{\sqrt{k+1}}.\] 
\end{corollary}

\section{Numerical Examples} 
\label{sec: experiment}
Distributed matrix completion aims to predict missing entries of a low rank target matrix using corrupted partial measurements distributed over a network \(\mathcal{G}=(\mathcal{V}, \mathcal{E})\) \cite{ling2012decentralized,wai2017decentralized}. Here we consider a template of the following form,
\begin{equation}\label{opt: matrix completion}
    \begin{array}{ll}
    \underset{x_1, \ldots, x_{|\mathcal{V}|}\in\mathbb{R}^{n\times n}}{\mbox{minimize}} & \sum_{i\in\mathcal{V}} \norm{o_i\odot (x_i-d)}_F^2\\
        \mbox{s.t.} & x_i=x_j, \enskip \forall\{ij\}\in\mathcal{E},\\
        & \norm{x_i}_*\leq \theta,\enskip  x_i=x_i^\top,\enskip \forall i\in\mathcal{V},\\
        &l_i\leq x_i\leq u_i,\enskip \forall i\in\mathcal{V}.
    \end{array}
\end{equation}
where \(\odot, \norm{\cdot}_F, \norm{\cdot}_*\) denote the Hadamard (entry-wise) product, Frobenius norm and nuclear norm, \(d\in\mathbb{R}^{n\times n}\) is a measurement of the target matrix. Further, for all \(i\in\mathcal{V}\), \(o_i\in\mathbb{R}^{n\times n}\) is a (0,1)-matrix whose sparsity pattern shows which measurements are available on node \(i\); \(l_i, u_i\in\mathbb{R}^{n\times n}\) are entry-wise upper and lower bound matrices. Note that the constraint \(\norm{x_i}_*\leq \theta\) with \(\theta>0\) aims to promote low rank solutions \cite{wai2017decentralized}. Let \(f_i(x_i)=\norm{o_i\odot(x_i-d)}_F^2\), \(X_i=\{x_i|\norm{x_i}_*\leq \theta, x_i=x_i^\top\}\) and \(Y_i=\{x_i|l_i\leq x_i\leq u_i\}\), then problem \eqref{opt: matrix completion} fits the template \eqref{opt: composite} in Section~\ref{sec: extension}\footnote{The vector space \(\mathbb{R}^n\) used in Section~\ref{sec: extension} can be extended to matrix space \(\mathbb{R}^{n\times n}\) by replacing vector inner product \(\langle x, y\rangle\) with Frobenius inner product \(\tr(x^\top y)\) and vector norm \(\norm{\cdot}_2\) with Frobenius norm \(\norm{\cdot}_F\).}. 

We consider an example of \eqref{opt: matrix completion} that estimates pairwise node distance based on partial noisy measurements on a random geometric graph as follows \cite{montanari2010positioning}. We first uniformly sample \(|\mathcal{V}|=10\) position vectors \(p_1, \ldots, p_{|\mathcal{V}|}\in [0, 1]^3\). Then define \(\mathcal
{G}\) by letting \(\{ij\}\in\mathcal{E}\) if \(\norm{p_i-p_j}_2\leq 0.6\). Let the \(ij\)-th entry of \(d\) be \(\norm{p_i-p_j}_2+\xi_{ij}\) where \(\xi_{ij}\) is sampled from the normal distribution with zero-mean and variance \(0.01\). For all \(i\in\mathcal{V}\), let the \(ij\)-th and \(ji\)-th entry be \(1\) if \(\{ij\}\in\mathcal{E}\) and zero elsewhere; let \(l_i\) to be a zero matrix; let off-diagonal entries in \(u_i\) be \(3\) and diagonal ones be \(0\). 

We test our algorithm~\eqref{alg: RC composite} on such example along with two benchmark methods, distributed projected gradient method (dist. Proj.) \cite{nedic2010constrained} and distributed conditional gradient method (dist. CG) \cite{wai2017decentralized}, see Figure~\ref{fig: convergence}. The convergence of our method is similar to that of the two benchmark methods \cite{nedic2010constrained,wai2017decentralized}. However, the benchmark methods use either linear minimization or projection over set \(X\cap Y\)--to our best knowledge, neither oracle admits efficient solution in our example. In comparison, each iteration of \eqref{alg: RC composite} uses linear minimization over \(X\), which can be computed very efficiently using Lanczo's algorithm (see \cite[Sec. 4.3]{jaggi2013revisiting} for a detailed discussion), and projection onto \(Y\), which amounts to computing entry-wise max/min. Hence the per-iteration computation of our method is much more efficient compared with methods in \cite{nedic2010constrained,wai2017decentralized}. The price for such efficiency is that, rather than ensuring \(x^k\in Y\), our method only ensures \(x^k\) converges to \(Y\) in the sense of Theorem~\ref{thm: 2}.

\begin{figure}
    \centering
    \includegraphics[width=\linewidth]{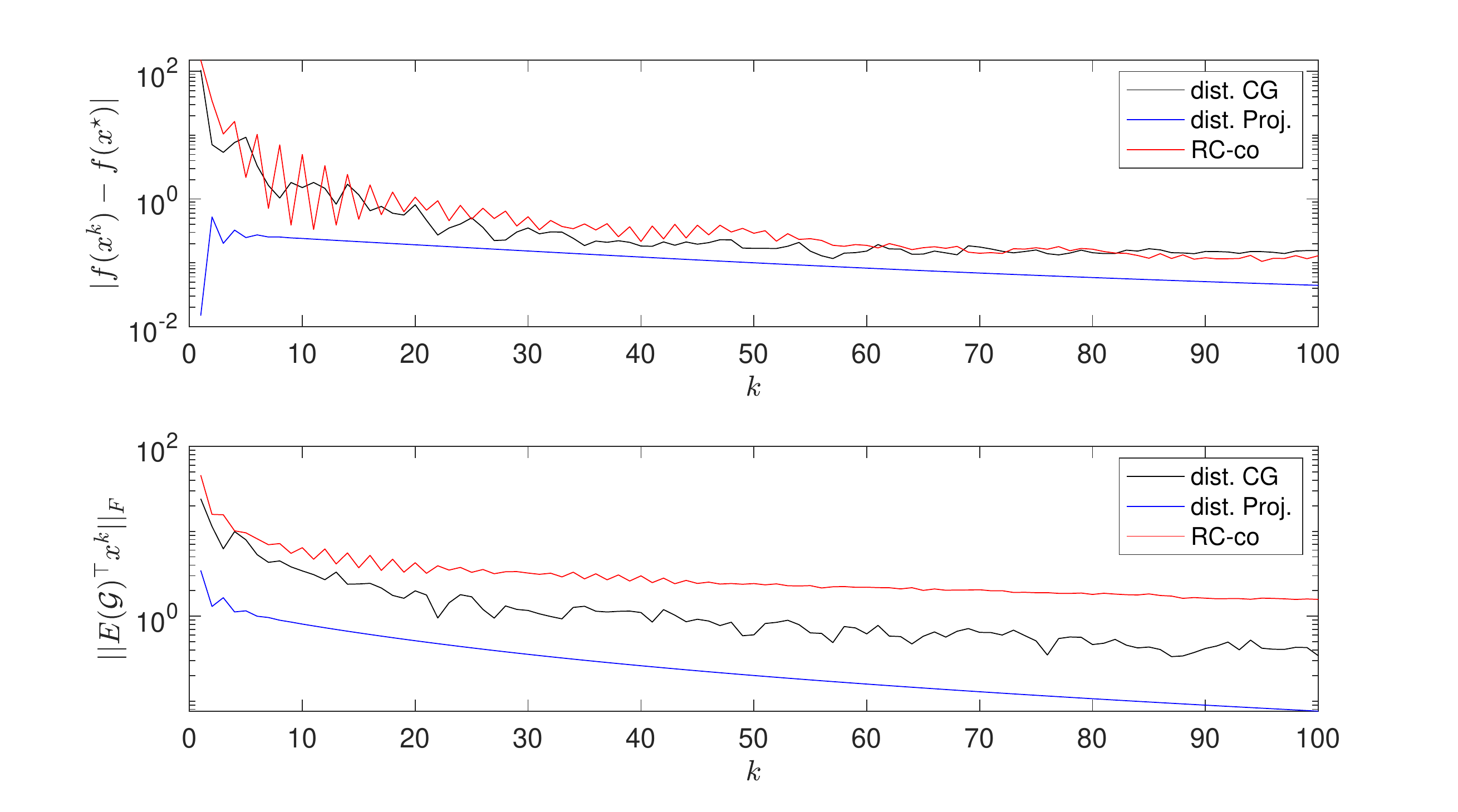}
    \caption{Convergence over iterations.}
    \label{fig: convergence}
\end{figure}

\section{Conclusion}
\label{sec: conclusion}
We propose a novel distributed conditional gradient method with \(O(1/\sqrt{k})\) convergence rate, and extend our results to composite constraints cases. However, our \(O(1/\sqrt{k})\) convergence still mismatches the \(O(1/k)\) convergence of the results in \cite{wai2017decentralized} and it is still unclear whether alternative circuits model such as RLC circuits \cite{yu2019rlc} can yield better algorithm design. Our future direction will focus on addressing these limitations, and non-convex extensions.

\section*{APPENDIX}
In this appendix, we first prove Theorem~\ref{thm: 2} and Corollary~\ref{cor: 2}, then prove Theorem~\ref{thm: 1} and Corollary~\ref{cor: 1} by letting \(Y=\mathbb{R}^{n|\mathcal{V}|}\). We start with the following lemma. 
\begin{lemma}\label{lem: 1} Under the assumptions of Theorem~\ref{thm: 2}, we have
\begin{equation*}
\begin{aligned}
 &f(x^k)-f(x^\star) \\
 \leq & \textstyle\frac{2\sigma^k\delta}{\sqrt{k}}-\textstyle\frac{\sqrt{k}}{2}(\norm{E(\mathcal{G})^\top x^k}_2^2+\norm{x^k-P_Y(x^k)}_2^2),
 \end{aligned}
\end{equation*}
where \(\sigma^k=\frac{\beta}{\sqrt{k}}+(\norm{L(\mathcal{G})}_2+1)r^0\).
\end{lemma}

\begin{proof}
Let 
\begin{equation}\label{eqn: distance definition}
    d(x)=\textstyle\frac{1}{2}\norm{x-P_Y(x)}_2^2, \enskip g(u)=\textstyle\frac{1}{2}\norm{u}_2^2+\max_{y\in Y}\langle u, y\rangle
\end{equation}
Since \(Y\) is convex and compact, one can show that \(d(x)\) is differentiable with \(\nabla d(x)=x-P_Y(x)\) \cite[Ex. 8.53]{rockafellar2009variational}.

In addition, observe that
\begin{equation*}
\begin{aligned}
    \textstyle\max_{u}\langle u, x\rangle-g(u)=&\textstyle\max_{u}-\frac{1}{2}\norm{u}_2^2-\max_{y\in Y}\langle u, y-x\rangle\\
    =&\textstyle \max_u\min_{y\in Y} -\frac{1}{2}\norm{u}_2^2+\langle u, x-y\rangle\\
    =&\textstyle \min_{y\in Y}\max_u -\frac{1}{2}\norm{u}_2^2+\langle u, x-y\rangle\\
    =&\textstyle \min_{y\in Y}\frac{1}{2}\norm{x-y}_2^2=d(x)
\end{aligned}\label{eqn: distance conjugate} 
\end{equation*}
where we swap the \(\max\) and \(\min\) since \(Y\) is bounded and \(-\frac{1}{2}\norm{u}_2^2+\langle u, x-y\rangle\) is convex in \(y\) and concave in \(u\) \cite[Cor.37.3.2]{rockafellar2015convex}. The above equation shows that \(d(x)\) is the conjugate of a \(1\)-strongly convex function \(g(u)\) \cite[Ex. 12.59 ]{rockafellar2009variational}, hence \(d(x)\) is convex and \(1\)-smooth \cite[Prop. 12.60]{rockafellar2009variational}, \ie, \(d(x)=\frac{1}{2}\norm{x-P_Y(x)}_2^2\) satisfies \eqref{eqn: smooth} with \(\beta=1\).

We define the following quantities.
\begin{subequations}
\begin{align}
   h(x)=&\textstyle\frac{1}{2}\norm{x-P_Y(x)}_2^2+\frac{1}{2}\norm{E(\mathcal{G})^\top x}_2^2,\label{eqn: lem1 def h}\\
    V^k=&\textstyle f(x^k)+r^{k-1}h(x^k)-f(x^\star), \label{eqn: lem1 def V}\\
    \Delta^k=&\textstyle \frac{1}{2}(\beta+(\norm{L(\mathcal{G})}_2+1)r^k)\delta.\label{eqn: lem1 def Delta}
\end{align}
\end{subequations}
Based on these definition, we first show the following
\begin{equation}
    \begin{aligned}
    &V^{k+1}-V^k-(r^k-r^{k-1})h(x^k)\\
    =& f(x^{k+1})+r^kh(x^{k+1})-(f(x^k)+r^kh(x^k))\\
    \leq &\textstyle\alpha^k\langle \nabla f(x^k)+r^k\nabla h(x^k), y^k-x^k\rangle\\
    &+\textstyle\frac{\beta+(\norm{L(\mathcal{G})}_2+1)r^k}{2}\norm{\alpha^k(y^k-x^k)}_2^2\\
    \leq &\textstyle\alpha^k\langle \nabla f(x^k)+r^k \nabla h(x^k), y^k-x^k\rangle+(\alpha^k)^2 \Delta^k
    \end{aligned}\label{eqn: lem1 eqn1}
\end{equation}
where the first inequality is an application of \eqref{eqn: smooth 1} to \((\beta+(\norm{L(\mathcal{G})}_2+1)r^k)\)-smooth function \(f(x)+r^k h(x)\); the second inequality is because \(\max_{x, x'\in X}\norm{x-x'}^2_2\leq \delta\). 

From the \(y\)-update in \eqref{alg: RC} we know that
\begin{equation}\label{eqn: lem1 eqn1.1}
    \langle \nabla f(x^k)+r^k\nabla h(x^k), y^k\rangle\leq \langle \nabla f(x^k)+r^k\nabla h(x^k), x^\star\rangle.
\end{equation}
Applying \eqref{eqn: convexity} to convex function \(f\) we can show
\begin{equation}
    \langle \nabla f(x^k), x^\star-x^k\rangle\leq f(x^\star)-f(x^k). \label{eqn: lem1 eqn1.2}
\end{equation}
Applying \eqref{eqn: smooth 2} to \(1\)-smooth function \(d(x)\) gives
\begin{equation}
    \begin{aligned}
        &\langle x^k-P_Y(x^k), x^\star-x^k\rangle \\
        \leq &\textstyle d(x^\star)-d(x^k)-\frac{1}{2}\norm{\nabla d(x^\star)-\nabla d(x^k)}_2^2=-2d(x^k).
    \end{aligned}\label{eqn: lem1 eqn1.3}
\end{equation}
where the last step is because \(\nabla d(x^\star)=x^\star-P_Y(x^\star)=0\), \(d(x^\star)=0\) and \(d(x)=\frac{1}{2}\norm{\nabla d(x)}_2\).
Further, since \(L(\mathcal{G})x^\star=0\), we have
\begin{equation}
    \langle L(\mathcal{G})x^k, x^\star-x^k\rangle=-\norm{E(\mathcal{G})^\top x^k}_2^2.\label{eqn: lem1 eqn1.4}
\end{equation}

Summing up \eqref{eqn: lem1 eqn1}, \(\alpha^k\times\)\eqref{eqn: lem1 eqn1.1}, \(\alpha^k\times\)\eqref{eqn: lem1 eqn1.2}, \(\alpha^kr^k\times\)\eqref{eqn: lem1 eqn1.3} and \(\alpha^kr^k\times\)\eqref{eqn: lem1 eqn1.4} gives the following
\begin{equation*}
\begin{aligned}
  &V^{k+1}-V^k-(r^k-r^{k-1})h(x^k)\\
  \leq &\alpha^k(f(x^\star)-f(x^k))-2\alpha^kr^kh(x^k)+(\alpha^k)^2\Delta^k
\end{aligned}
\end{equation*}
Rearranging terms and use \eqref{eqn: lem1 def V}, we have
\begin{equation}
\begin{aligned}
    & V^{k+1}-(1-\alpha^k)V^k\\
    \leq &((1-\alpha^k)(r^k-r^{k-1})-\alpha^kr^k)h(x^k)+(\alpha^k)^2\Delta^k
\end{aligned}\label{eqn: lem1 eqn2}
\end{equation}
Since \(\alpha^k=\frac{2}{k+1}\), \(r^k=r^0\sqrt{k+1}\), 
\begin{equation}
\begin{aligned}
    &(1-\alpha^k)(r^k-r^{k-1})-\alpha^kr^k\\
    \leq & (1-\alpha^k)(r^k-r^{k-1})- \alpha^kr^{k-1}\\
    =&\textstyle\frac{\big(k-1-\sqrt{k(k+1)}\big)r^0}{\sqrt{k+1}}<\textstyle\frac{-r^0}{\sqrt{k+1}}< 0
\end{aligned}\label{eqn: lem1 eqn2.1}
\end{equation}
Since \(h(x^k)\geq 0\), substituting \eqref{eqn: lem1 eqn2.1} into \eqref{eqn: lem1 eqn2} gives
\begin{equation*}
    V^{k+1}\leq(1-\alpha^k)V^k+(\alpha^k)^2\Delta^k
\end{equation*}
Using this recursion \(k\) times, we can obtain the following 
\begin{equation}
\begin{aligned}
    V^{k+1}\leq & V^1\textstyle \prod\limits_{s=1}^k(1-\alpha^s)+(\alpha^k)^2\Delta^k\\
    &+\textstyle\sum\limits_{s=1}^{k-1}\big[(\alpha^s)^2\Delta^s\prod\limits_{m=s}^{k-1}(1-\alpha^{m+1})\big]\\
    \leq &\textstyle(\alpha^k)^2\Delta^k+\Delta^k\sum\limits_{s=1}^{k-1}\big[(\alpha^s)^2\prod\limits_{m=s}^{k-1}(1-\alpha^{m+1})\big]
\end{aligned}\label{eqn: lem1 eqn3}
\end{equation}
where the last step is because \(1-\alpha^1=0\) and \(\Delta^s\leq \Delta^k\) for all \(s\leq k\), due to \eqref{eqn: lem1 def Delta}. Finally, since \(\alpha^k=\frac{2}{k+1}\), we have \(\prod_{m=s}^{k-1}(1-\alpha^{m+1})=\frac{s(s+1)}{k(k+1)}\) and
\begin{equation*}
    \textstyle\sum\limits_{s=1}^{k-1}\big[(\alpha^s)^2\prod\limits_{m=s}^{k-1}(1-\alpha^m)\big]=\textstyle\sum\limits_{s=1}^{k-1}\frac{4}{(s+1)^2}\frac{s(s+1)}{k(k+1)}< \frac{4(k-1)}{k(k+1)}.
\end{equation*}
Substituting the above inequality into \eqref{eqn: lem1 eqn3} gives
\[\textstyle V^{k+1}\leq \frac{4}{k+1}\big(\frac{1}{k+1}+1-\frac{1}{k}\big)\Delta^k\leq \frac{4}{k+1}\Delta^k,\]
which, combined with \eqref{eqn: lem1 def V} and \eqref{eqn: lem1 def Delta}, completes the proof.
\end{proof}

\emph{Proof of Theorem~\ref{thm: 2}}
Since \eqref{alg: RC composite} ensures that \(x^k\in X\) for all \(k\),  we can use \eqref{kkt composite: dual} and \eqref{def: normal cone} to show that 
\begin{equation*}
    \begin{aligned}
    0\leq &\langle E(\mathcal{G})u^\star+\nabla f(x^\star)+v^\star, x^k-x^\star\rangle\\
    \leq &f(x^k)-f(x^\star)+\langle u^\star, E(\mathcal{G})^\top x^k\rangle+\langle v^\star, x^k-x^\star\rangle
    \end{aligned}
\end{equation*}
where the second step is obtained using \eqref{kkt composite: primal} and \eqref{eqn: convexity}. In addition, since \(v^\star\in N_Y(x^\star)\), we know that
\begin{equation*}
    \langle v^\star, P_Y(x^k)-x^\star\rangle\leq 0
\end{equation*}
Summing up the above two inequalities we have
\begin{equation}
-\langle \begin{bmatrix}
u^\star\\
v^\star
\end{bmatrix}, \begin{bmatrix}
E(\mathcal{G})^\top x^k\\
x^k-P_Y(x^k)
\end{bmatrix}\rangle\leq f(x^k)-f(x^\star) \label{eqn: thm2 eqn1}
\end{equation}
Further, using Cauchy-Schwartz inequality we can show
\begin{equation}
\langle \begin{bmatrix}
u^\star\\
v^\star
\end{bmatrix}, \begin{bmatrix}
E(\mathcal{G})^\top x^k\\
x^k-P_Y(x^k)
\end{bmatrix}\rangle\leq \norm{\begin{bmatrix}
u^\star\\
v^\star
\end{bmatrix}}_2\norm{\begin{bmatrix}
E(\mathcal{G})^\top x^k\\
x^k-P_Y(x^k)
\end{bmatrix}}_2\label{eqn: thm2 eqn2}
\end{equation}

Summing up \eqref{eqn: thm2 eqn1}, \eqref{eqn: thm2 eqn2}, and inequality in Lemma~\ref{lem: 1} gives
\begin{equation}
    \textstyle -\frac{\sqrt{k}}{2}\omega^2+\rho \omega + \frac{2\sigma^k \delta}{\sqrt{k}}\geq 0,\enskip \omega\geq 0.
\end{equation}\label{eqn: thm2 eqn3}
where 
\[\rho = \norm{\begin{bmatrix}
u^\star\\
v^\star
\end{bmatrix}}_2, \enskip \omega= \norm{\begin{bmatrix}
E(\mathcal{G})^\top x^k\\
x^k-P_Y(x^k)
\end{bmatrix}}_2\]
Solving this quadratic inequality in terms of \(\omega\) gives
\begin{equation}
    \begin{aligned}
    0\leq \omega=&\norm{\begin{bmatrix}
E(\mathcal{G})^\top x^k\\
x^k-P_Y(x^k)
\end{bmatrix}}_2\\
\leq & \textstyle  \frac{1}{\sqrt{k}}\big(\rho+\sqrt{\rho^2+4\sigma^k \delta}\big)\leq  \frac{2}{\sqrt{k}}\big(\rho+\sqrt{\sigma^k\delta}\big)
    \end{aligned}\label{eqn: thm2 eqn4}
\end{equation}
where the last step is because \(\sqrt{a^2+b^2}\leq a+b\) for any \(a, b\geq 0\). This proves the second inequality.

Next, substituting \eqref{eqn: thm2 eqn4} into the sum of \eqref{eqn: thm2 eqn1} and \eqref{eqn: thm2 eqn2} gives 
\begin{equation}
    \textstyle -\frac{2}{\sqrt{k}}\rho(\rho+\sqrt{\sigma^k\delta})\leq f(x^k)-f(x^\star).
    \label{eqn: thm2 eqn5}
\end{equation}
Finally, Lemma \ref{lem: 1} directly implies that \(f(x^k)-f(x^\star)\leq \frac{2\sigma^k \delta}{\sqrt{k}}\). Combine this with \eqref{eqn: thm2 eqn5} gives the first inequality.
\qed

\emph{Proof of Corollary~\ref{cor: 2}} If the exact optimal solution in \eqref{alg: RC composite} is replaced by an \(\epsilon^k\)-optimal solution with \(\epsilon^k\leq \kappa\big(\frac{\beta}{\sqrt{k+1}}+(\norm{L(\mathcal{G}}_2+1)r^0\big)\frac{\delta}{\sqrt{k+1}}\), we need to replace \eqref{eqn: lem1 eqn1.1} in the proof of Lemma~\ref{lem: 1} with the following
\begin{equation*}
\begin{aligned}
    \langle \nabla f(x^k)+r^k\nabla h(&x^k), y^k\rangle\leq \langle \nabla f(x^k)+r^k\nabla h(x^k), x^\star\rangle\\
    &+\textstyle \kappa\big(\frac{\beta}{\sqrt{k+1}}+(\norm{L(\mathcal{G}}_2+1)r^0\big)\frac{\delta}{\sqrt{k+1}},
\end{aligned}
\end{equation*}
where \(h(x)\) is defined as in \eqref{eqn: lem1 def h}. Following the rest of the proof of Lemma~\ref{lem: 1} and Theorem~\ref{thm: 2} completes the proof. \qed

\emph{Proof of Theorem~\ref{thm: 1} and Corollary~\ref{cor: 1} } If  \(Y=\mathbb{R}^{|\mathcal{V}|n}\), then \(x-P_Y(x)=0\) and \(N_Y(x)=\{0\}\) for all \(x\). Hence letting \(x^k-P_Y(x^k)=0\) and \(v^\star=0\) in the proof of Theorem~\ref{thm: 2} yields the proof of Theorem~\ref{thm: 1}. Notice that in this case \(h(x)\) in \eqref{eqn: lem1 def h} is \(\norm{L(\mathcal{G})}_2\)-smooth rather than \((\norm{L(\mathcal{G})}_2+1)\)-smooth, causing a change in \(\sigma^k\). The proof of Corollary~\ref{cor: 1} is similar. \qed




\bibliographystyle{IEEEtran}
\bibliography{IEEEabrv,reference}

\end{document}